\documentclass[10pt]{amsart}
\usepackage[cp866]{inputenc}
\usepackage[english,russian]{babel}
\usepackage{amsmath}
\usepackage{amssymb}
\usepackage{amsfonts}

\newtheorem{lemma}{Lemma}
\newtheorem{theorem}{Theorem}

\newtheorem{corollary}{Corollary}
\newtheorem{proposition}{Proposition}

\newcommand{\Aut}{\mathrm{Aut}}
\newcommand{\rank}{\mathrm{rank}}
\newcommand{\wt}{\mathrm{wt}}

\newcommand{\GL}{\mathrm{GL}}
\newcommand{\GA}{\mathrm{GA}}
\newcommand{\Ker}{\mathrm{Ker}}

\begin{document}
\renewcommand{\refname}{References}
\renewcommand{\proofname}{Proof.}
\renewcommand{\tablename}{Table}
\renewcommand{\dim}{\mathrm{dim}}

\title[A concatenation construction for propelinear perfect codes
]{A concatenation construction for propelinear perfect codes
}
\author{{I.Yu.Mogilnykh, F. I. Solov'eva}}%

\address{Ivan Yurevich Mogilnykh
\newline\hphantom{iii} Tomsk State University, Regional Scientific and Educational Mathematical Center,
\newline\hphantom{iii} pr. Lenina, 36,
\newline\hphantom{iii}634050, Tomsk, Russia,
\newline\hphantom{iii} Sobolev Institute of Mathematics,
\newline\hphantom{iii} pr. Koptyuga, 4,
\newline\hphantom{iii} 630090, Novosibirsk, Russia}%
\email{ivmog@math.nsc.ru}

\address{Faina I. Solov'eva
\newline\hphantom{iii} Sobolev Institute of Mathematics,
\newline\hphantom{iii} pr. ac. Koptyuga 4,
\newline\hphantom{iii} 630090, Novosibirsk, Russia}%
\email{sol@math.nsc.ru}%

\thanks{\copyright \ 2019  I.Yu.Mogilnykh, F. I. Solov'eva}
\thanks{\rm The work was supported by the Ministry of Education and Science of Russia (state assignment No.
1.13557.2019/13.1).}



\maketitle

\begin{quote}
{\small \noindent{\sc Abstract.} A code $C$ is called propelinear
if there is a subgroup of $\Aut(C)$ of order $|C|$ acting
transitively on the codewords of $C$. In the paper new propelinear
perfect binary codes of any admissible length more than $7$ are
obtained by a particular case of the Solov'eva concatenation
construction--1981 and the regular subgroups
of the general affine group
 of the vector space over $\mbox{GF}(2)$.

\medskip

}
\end{quote}

\section{Introduction}

 The vector space of dimension $n$ over the
Galois field $F$ of two elements  with respect to the Hamming
metric is denoted by $F^n$.
 The {\it Hamming distance} between
any two vectors $x,y\in F^n$ is defined as the number of
coordinates in which $x$ and $y$ differ. {\it The support} of a
vector $x$ from $F^n$ denoted by $\mbox{supp}(x)$ is the
collection of the indices of its nonzero coordinate positions.
{\it The Hamming weight} $\wt(x)$
 of a vector $x$ is the size of its support.   A  {\it code} of length $n$
is an arbitrary set of vectors of $F^n$ that are called {\it
codewords} of $C$.
The {\it code distance} of a code is the minimum
value of the Hamming distance between two different codewords from
the code. A code $C$ is called {\it perfect binary
single-error-correcting}
 (briefly  {\it perfect}) if for any vector $x$ from  $F^n$  there
 exists exactly one vector $y
\in C$ at the Hamming distance not more than 1 from the vector
$x$. A perfect linear code is called the {\it Hamming code}.
Adding the overall parity check to all codewords of a code of
length $n$ we obtain the code of length $n+1$ that is called {\it extended}.

Let $x$ be a binary vector of $F^n$, $\pi$ be a permutation of the
coordinate positions of vectors in $F^n$. Consider the
transformation $(x,\pi)$ that maps a binary vector $y$ as follows:
 $$(x,\pi)(y)=x+\pi(y),$$
where $\pi(y)=(y_{\pi^{-1}(1)},\ldots,y_{\pi^{-1}(n)})$.
 The composition of two
transformations $(x,\pi)$, $(y,\pi')$ is defined as
$$(x,\pi)\cdot(y,\pi')=(x+\pi(y),\pi\circ\pi'),$$
where $\pi \circ \pi'$ is the composition of permutations $\pi$
and $\pi'$ defined as follows:
 $$ \pi\circ\pi'(i)=\pi(\pi'(i))$$
 for any $i\in \{1,2,\ldots,n\}.$  {\it The automorphism group} $\Aut(F^n)$ of $F^n$ is the
group of all such transformations $(x,\pi)$ with respect to the
composition. {\it The automorphism group} $\Aut(C)$ of a code $C$
is the setwise stabilizer of $C$ in $\Aut(F^n)$.

A group $G$ acting on a
set is called {\it regular} if the action is transitive and the
order of $G$ coincides with the size of the set. A subgroup of the
automorphism group of a code is called {\it regular} if it acts
regularly on  the set of its codewords. A code $C$ is called
{\it propelinear} \cite{PR} if $\Aut(C)$ has a regular
subgroup.

 It is well-known that the supports of the  codewords of weight 3
in any perfect code containing the all-zero vector form a Steiner
triple system. A perfect code $C$ is called {\it homogeneous} if
all Steiner triple systems of the codes $C+y$, $y\in C$ are
isomorphic. The homogeneous perfect codes were introduced  in
\cite{OPP}. Obviously, any propelinear code is necessarily
homogeneous. Despite of the existence of nonpropelinear
homoge\-neous Vasil'ev perfect codes for any length $n$, $n\geq
15$ \cite{MS2}, the existence of a rich construction of such codes
remains to be an open problem.

 Propelinear codes play an important role in the
theory of optimal codes since
  they are close to linear codes by several properties of their automorphism
  groups.
  Nowadays there are known several classes of propelinear codes, among
  them are Preparata and Kerdock codes \cite{HKCSS}, \cite{RP} $Z_4$-linear Reed-Muller codes \cite{PRS},
 $Z_4$-linear and $Z_{2^k}$-linear Hadamard codes \cite{K1}
\cite{K2}, etc.

Classical propelinear perfect codes are $Z_4$-linear \cite{K1} and
$Z_2Z_4$-linear codes \cite{BR}. It is known that propelinear
perfect codes can be obtained by the Plotkin and the Vasil'ev
constructions \cite{BMRS}. In
\cite{BorgesMogilnykhRifaSoloveva} all transitive codes from
\cite{Pot}
 found by a representation via the Phelps construction were proved to be propelinear.
The codes from \cite{Pot} were later generalized by Krotov and
Potapov in \cite{KrotovPotapov} who utilized quadratic functions
in the Vasil'ev construction. Note that the Vasil'ev construction
was generalized by the Mollard construction for propelinear
perfect codes \cite{BMRS}. The approaches of
\cite{Pot,BorgesMogilnykhRifaSoloveva,KrotovPotapov} gave
exponential classes of propelinear codes (the best lower bound was
obtained in \cite{KrotovPotapov}) but all these codes are of small
rank $n-\log(n+1)+1$, where $n$ is the length of the codes.
Moreover, the ranks of $Z_4$-linear extended perfect codes of
length $n$ do not exceed $n-\lceil \frac{\log(n)+1}{2}\rceil$, see
\cite{K1} (an analogous result \cite{PR} holds for ranks of
$Z_2Z_4$-linear perfect codes from \cite{BR}).

 The question of finding propelinear
perfect codes of large ranks was considered in \cite{GMS} and was
based again on the Mollard construction. In this paper a solution for the rank problem for propelinear codes is given with exception of few finite open cases.
Therefore the problem of
finding new methods of constructing propelinear non-Mollard codes
of large ranks is open.
 The kernel problem as far as the rank and kernel problem
is still open for propelinear perfect binary codes. Recall that
the rank and kernel problem for perfect binary codes was solved in
the paper \cite{AHS}.


In the paper we obtain a  new class of propelinear perfect and
extended perfect binary codes of
  ranks in $\{2^r-r-1,\ldots,2^r-2\}$ and  the dimensions of the kernels in $\{2^r-2r-2,\ldots,2^r-r-3\}$. The paper is organized as follows. The general construction is given in Section 2.
  The concatenation construction  \cite{Sol1981}, see also \cite{Phelps}, uses a partition of the even weight code into the extended perfect codes of length $2^r$ and
a permutation on the elements of the partition. It is not
difficult to show that the  full rank codes can not be obtained by
the construction \cite{Sol1981}. In the paper we consider the case
when the partition is into extended Hamming codes. The
construction in Section 2 for propelinear codes uses the
automorphisms of the regular subgroups of the general affine group
 of $F^r$ as permutations. In Section
 3 we investigate ranks and kernels of the version of the construction \cite{Sol1981} with arbitrary permutations.
We obtain the expressions for the ranks and the dimension of the kernels of these codes in terms of these permutations.
Moreover, we show that any such code with the dimension of the kernel $2^r-2r-2$ is inequivalent to any Mollard
 propelinear code. The discussion is continued in  Section 4 where we construct an
infinite series of new propelinear perfect codes.
For
this purpose we apply the direct product construction for regular
subgroups of the general affine group \cite{Mog} to a regular
subgroup of the general affine group constructed by Hegedus in
\cite{H}.

\section{A construction for  propelinear perfect codes}

Let the coordinates of the vector space $F^{2^r}$ be indexed by
the vectors from $F^r$. Below the all-zero vector  ${\bf 0}^r$ of length $r$ is denoted by
${\bf 0}$ and the length of the vector will be always clear from
the context. Define the following code:
 $${\mathcal H}=\{c\in F^{2^r}:\sum_{a: c_a=1}a={\bf 0}, \wt(c)\equiv 0(\mbox{mod }2)\}.$$

 Given a code $C$ and a coordinate position {\it the
punctured code} $C'$ is defined as the code whose codewords are
obtained by removing the coordinate in all codewords of
$C$. Consider the code ${\mathcal H}'$obtained by puncturing ${\mathcal H}$  in the
coordinate indexed by ${\bf 0}$. We index the coordinate positions of $F^{2^r-1}$
by the nonzero vectors of $F^r$ and therefore we have:
$${\mathcal H}'=\{c\in F^{2^r-1}:\sum_{a:
c_a=1}a={\bf 0}\}.$$

 For an arbitrary vector  $a$ in $F^r$ the code
${\mathcal H}+e_a+e_{{\bf 0}}$ is denoted by ${\mathcal H}_a$,
here $e_a$ is the vector in $F^n$ with the only one nonzero
position $a$, $a\in F^r$. The code ${\mathcal H}$ is an extended
Hamming code and the collection of the cosets ${\mathcal H}_a$,
where $a\in F^r$, is the partition of the set of all even weight
vectors of $F^{2^r}$ into cosets of the code ${\mathcal H}$.

Denote the {\it general linear group} that consists of nonsingular
$r\times r$ matrices  over  $F$  by  $\GL(r,2)$. Consider an
affine transformation $(a,M)$, $a\in F^r, M \in \GL(r,2)$. Its
action on $F^r$ is defined as \begin{equation}\label{actionGA}(a,
M)(b) = a + Mb,\end{equation} $b \in F^r$. The composition of any
two affine transformations $(a,M)$ and $(b,M')$ is the
transformation $(a+bM,MM')$. The {\it general affine group} of the
space $F^r$ with elements $\{(a,M):a \in F^r, M\in \GL(r,2)\}$
with respect to the composition is denoted by $\GA(r,2)$.

A subgroup $G$ of a group $\GA(r,2)$ is called {\it regular} if it
is regular with respect to the action (\ref{actionGA}) on the
vectors of $F^r$. The action of $\GA(r,2)$ on the vectors of $F^r$
is equivalent to the action of the automorphism group on the
codewords of the Hadamard code (the dual code to the Hamming code
of length $2^r-1$), see \cite{Mog}. Therefore the regular
subgroups of $\GA(r,2)$ are in a one-to-one correspondence with the
regular subgroups of the automorphism group of the Hadamard code.

By definition for any regular subgroup $G$ of the group $\GA(r,2)$
and any $a, a\in F^r$ there is a unique affine transformation that
maps ${\bf 0}$ to $a$.
In throughout what follows we denote it
by $g_a$. Obviously, $g_a$ is $(a,M_a)$ for some matrix  $M_a$ in
$\GL(r,2)$. Since \begin{equation}\label{gmult}
g_ag_b=(g_a(b),M_aM_b)\end{equation} we have

\begin{equation}\label{Mid0}
M_{g_a(b)}=M_aM_b.
\end{equation}

Let  $T$ be an  automorphism of a regular subgroup  $G$ of the
group $\GA(r,2)$. By $\tau$ we denote the permutation on the
vectors of $F^{r}$ {\it induced} by the
 automorphism  $T$, i.e.
$$T(g_a)=g_{\tau(a)}.$$
Obviously we always have $\tau({\bf 0})={\bf 0}$. Since  $T$ is an
automorphism of $G$ then by the definition of $\tau$ and (\ref{gmult}) we have
$(\tau(g_a(b)),M_{\tau(g_a(b))})=T(g_ag_b)=T(g_a)T(g_b)=g_{\tau(a)}g_{\tau(b)}=(g_{\tau(a)}(\tau(b)),M_{\tau(a)}M_{\tau(b)}).$

\noindent
Therefore the following equalities hold:
\begin{equation}\label{gid}
\tau(g_a(b))=g_{\tau(a)}(\tau(b)),
\end{equation}
\begin{equation}\label{Mid}
M_{\tau(g_a(b))}=M_{\tau(a)}M_{\tau(b)}.
\end{equation}




The concatenation of two vectors $x\in F^{r'}$ and $y\in F^{r''}$
is denoted by $x|y$. For codes $C$ and $D$ by $C\times D$ denote
the code $\{x|y:x\in C, y \in D\}$. Let $\pi'$, $\pi''$ be
permutations on the vectors of $F^{r'}$ and $F^{r''}$ respectively
then by $\pi'|\pi''$ we denote the permutation acting on the
concatenations $x|y$ of the vectors $x\in F^{r'}$ and $y\in
F^{r'}$ from $F^{r'+r''}$ as follows:
$(\pi'|\pi'')\,(x|y)=\pi'(x)|\pi''(y)$.

In particular, let  $\pi'$ and $\pi''$ be permutations on the
coordinate positions of $F^{r'}$. A permutation on the coordinates
of the vector space naturally induces the permutation on the set
of vectors. So throughout  Section 2, we use the same notation
$\pi'|\pi''$ for the permutation of the coordinate positions of
$F^{2r'}$ that acts as follows:
$(\pi'|\pi'')\,(x|y)=\pi'(x)|\pi''(y)$, for any $x,y\in F^{r'}$.

Consider the following  particular case of the concatenation
construction  \cite{Sol1981} for extended perfect codes:
\begin{equation}\label{Soloveva_codes} S_{{\mathcal
H},\tau}=\bigcup_{a\in F^r} {\mathcal H}_a\times {\mathcal
H}_{\tau(a)},\end{equation} where $\tau$ is a bijection from $F^{2^r}$
to $F^{2^r}$.
\begin{theorem}\label{Theorem 1}Let $G$  be a regular subgroup of $\GA(r,2)$ and $\tau$ be the permutation induced by an automorphism of $G$.
Then the code $S_{{\mathcal H},\tau}$ is a propelinear extended
perfect binary code of length $2^{r+1}$.
\end{theorem}

\begin{proof} 
 For an element $g_a=(a,M_a)$ of a
regular subgroup $G$ of $\GA(r,2)$ by $\pi_a$ we denote the
permutation corresponding to the matrix $M_a$, i.e.
\begin{equation}\label{pi_and_M} \pi_a(b)=M_ab.
\end{equation}

Since $M_a$ is in $\GL(r,2)$, it preserves linear independency, so by definition of ${\mathcal H}$, we have
\begin{equation}
\label{eq00}
\pi_a({\mathcal H})={\mathcal H}, \pi_a(e_{\bf 0})=e_{\bf 0}.
\end{equation}

Consider the following set of automorphisms of   $F^{r+1}$:
$$\Gamma=\bigcup_{a\in F^r}\{(x|y,\pi_{a}|\pi_{\tau(a)}): x\in {\mathcal H}_a,y\in
{\mathcal H}_{\tau(a)}\}.$$

Note that when $(x|y,\pi_{a}|\pi_{\tau(a)})$ runs through
$\Gamma$, the vector $x|y$ runs through the code $S_{{\mathcal
H},\tau}$. If we prove that $\Gamma$ is a group, then the
orbit of the all-zero vector from $F^{2^{r+1}}$ under $\Gamma$ is
$S_{{\mathcal H},\tau}$, so $\Gamma$ is a regular subgroup of $
\Aut(S_{{\mathcal H},\tau})$.

We now show that $\Gamma$ is closed under composition. Consider
two automorphisms $(x|y,\pi_{a}|\pi_{\tau(a)})$ and
$(u|v,\pi_{b}|\pi_{\tau(b)})$ from $\Gamma$, by definition of
$\Gamma$ we have:
\begin{equation}\label{eq01}
x \in {\mathcal H}_a, y \in {\mathcal H}_{\tau(a)},
\end{equation}
\begin{equation}\label{eq02}
u \in {\mathcal H}_b, v \in {\mathcal H}_{\tau(b)}.
\end{equation}
 We denote the composition of these two automorphisms by
$(w|z,\pi_a\pi_b|\pi_{\tau(a)}\pi_{\tau(b)})$, where
$$w=x+\pi_{a}(u),$$ $$z=y+ \pi_{\tau(a)}(v)$$
and show that it is in $\Gamma$.

Since the code ${\mathcal H}$ is linear and  $\pi_a({\mathcal
H})={\mathcal H}$ (see (\ref{eq00})), $u\in {\mathcal H}_b$,
${\mathcal H}_b=e_b+e_{\bf 0}+{\mathcal H}$ (see (\ref{eq02})) we
have
 $$w+{\mathcal H}=x+\pi_a(u)+{\mathcal H}=x+\pi_a(e_b+e_{\bf 0})+{\mathcal H}.$$
 Because  $x$ is in ${\mathcal H}_a={\mathcal H}+e_a+e_{{\bf 0}}$ (see (\ref{eq01})), $\pi_a(e_{{\bf 0}})=e_{{\bf
 0}}$ (see (\ref{eq00})) we obtain   $x+\pi_a(e_b+e_{\bf 0})+{\mathcal H}=e_a+\pi_a(e_b)+{\mathcal
H}$.
  By (\ref{pi_and_M}), i.e.
 $\pi_a(e_b)=e_{M_ab}$, we have  $e_a+\pi_a(e_b)+{\mathcal
H}=e_a+e_{M_ab}+{\mathcal H}$. The linear code ${\mathcal H}$
contains the vector $e_a+e_{M_ab}+e_{a+M_ab}+e_{{\bf 0}}$, hence
\begin{equation}\label{eq1}w+{\mathcal H}={\mathcal H}_{a+M_ab}={\mathcal H}_{g_a(b)}.
\end{equation}
\noindent
 Let us now show that
\begin{equation}\label{eq2}z+{\mathcal H}={\mathcal H}_{\tau({g_a(b)})}.\end{equation}
\noindent
 From  $z=y+ \pi_{\tau(a)}(v)$, $\pi_{\tau(a)}({\mathcal H})={\mathcal H}$,
$\pi_{\tau(a)}(e_{\bf 0})=e_{\bf 0}$ (see (\ref{eq00})), $v\in {\mathcal
H}_{\tau(b)},y \in {\mathcal H}_{\tau(a)}$ (see (\ref{eq01}) and (\ref{eq02})) and
by (\ref{pi_and_M}) we have
$$z+{\mathcal H}=y+\pi_{\tau(a)}(v)+{\mathcal H}=e_{\tau(a)}+\pi_{\tau(a)}(e_{\tau(b)})+{\mathcal H}=e_{\tau(a)}+e_{M_{\tau(a)}\tau(b
)}+{\mathcal H}.$$ \noindent By definition of ${\mathcal H}$ we
have $e_{\tau(a)}+e_{M_{\tau(a)}\tau(b
)}+e_{\tau(a)+M_{\tau(a)}\tau(b)}+e_{\bf 0}\in {\mathcal H}$,
which combined with $g_{\tau(a)}(\tau(b
))=\tau(a)+M_{\tau(a)}\tau(b )$ and the fact that ${\mathcal H}$
is linear we obtain $$e_{\tau(a)}+e_{M_{\tau(a)}\tau(b
)}+{\mathcal H}={\mathcal H}_{\tau(a)+M_{\tau(a)}\tau(b )}
={\mathcal H}_{g_{\tau(a)}(\tau(b))}.$$

Using (\ref{gid}) we have $z+{\mathcal H}={\mathcal
H}_{\tau(g_a(b))}$, i.e. (\ref{eq2}) holds.

Note that according to the correspondence (\ref{pi_and_M}) the
equalities $M_{g_a(b)}=M_aM_b$ and
$M_{\tau(g_a(b))}=M_{\tau(a)}M_{\tau(b)}$, see (\ref{Mid0}) and
(\ref{Mid}), can be rewritten as $\pi_{g_a(b)}=\pi_a\pi_b$ and
$\pi_{\tau(g_a(b))}=\pi_{\tau(a)}\pi_{\tau(b)}$. These equalities
imply that the permutation $\pi_a\pi_b|\pi_{\tau(a)}\pi_{\tau(b)}$
is equal to $\pi_{g_a(b)}| \pi_{\tau(g_a(b))}$. Therefore the
considered composition of automorphisms
$(x|y,\pi_{a}|\pi_{\tau(a)})$ and $(u|v,\pi_{b}|\pi_{\tau(b)})$,
i.e. $(w|z,\pi_{g_a(b)}| \pi_{\tau(g_a(b))})$ belongs to $\Gamma$
since by the equalities (\ref{eq1}) and (\ref{eq2}) the vector
$w|z$ belongs to ${\mathcal H}_{g_a(b)}\times {\mathcal
H}_{\tau(g_a(b))}$. Hence $\Gamma$ is a regular subgroup
 of the automorphism group of the code $S_{{\mathcal H},\tau}$ and
  the code $S_{{\mathcal H},\tau}$ is propelinear.

\end{proof}

\begin{proposition}\label{Prop1}
Let $C$ be a propelinear code with minimum distance at least $2$
whose automorphism group contains a regular subgroup $G$. Let $i$
be a coordinate such that $\pi(i)=i$ for any $(x,\pi) \in G$. Then
the code $C'$ obtained from $C$ by puncturing in the $i$th
coordinate position is propelinear.

\end{proposition}
\begin{proof}

For $x\in C$ let $x'$ denote the codeword of $C'$ obtained by
deleting its $i$th coordinate position. Suppose the coordinates of
$C'$ are indexed by the coordinates of the code $C$ without $i$th
position, so $\mbox{supp}(x')=\mbox{supp}(x)\setminus \{i\}$ if $i \in \mbox{supp}(x)$
and $\mbox{supp}(x')=\mbox{supp}(x)$ otherwise. For a permutation $\pi$, where
$(x,\pi)\in G$, by $\pi'$ denote the permutation acting on the
coordinate positions of $C'$, where $\pi'(j)=\pi(j)$, for any
coordinate $j$ of the code $C$ different from $i$. Obviously, the
group $G'=\{(x',\pi'):(x,\pi)\in G\}$ is isomorphic to $G$ and
$G'$ is a regular subgroup of $\Aut(C')$.

\end{proof}

 Consider the following puncturing $S_{{\mathcal
H},\tau}'$ of $S_{{\mathcal H},\tau}$:
\begin{equation}\label{code_S} S_{{\mathcal
H},\tau}'=\bigcup_{a\in F^r} {\mathcal H}'_a\times {\mathcal
H}_{\tau(a)},\end{equation} where ${\mathcal H}'_a=e_a+{\mathcal
H}',$ ${\mathcal H}'_{\bf 0}={\mathcal H}'.$ The code
$S_{{\mathcal H},\tau}'$ is
  perfect. Let $\tau$ be a permutation induced by an automorphism of a regular subgroup $G$ of $\GA(r,2)$.
In this case for every $a\in F^r$ the permutation
$\pi_a$ defined in (\ref{pi_and_M}) fixes the coordinate ${\bf 0}$ of ${\mathcal H}$. Then
every permutation $\pi_a|\pi_{\tau(a)}$ of any automorphism
$(x|y,\pi_a|\pi_{\tau(a)})$ of the regular subgroup $\Gamma$ of
$\Aut(S_{{\mathcal H},\tau})$ fixes the coordinate position of
$S_{{\mathcal H},\tau}$ in which we puncture the code
$S_{{\mathcal H},\tau}$ to obtain $S_{{\mathcal H},\tau}'$. By
Proposition \ref{Prop1} we see that $S_{{\mathcal H},\tau}'$ is
propelinear. Therefore,
a class of propelinear perfect codes is obtained: 

\begin{corollary}\label{Rem1}Let $G$  be a regular subgroup   of $\GA(r,2)$ and $\tau$ be the permutation induced by an automorphism of the group $G$.
Then the code $S_{{\mathcal H},\tau}'$ is a propelinear perfect
binary code of length $2^{r+1}-1$.
\end{corollary}

Moreover, the values for invariants (i.e. rank and kernel) which
we obtain below in the paper for  the extended perfect code
$S_{{\mathcal H},\tau}$ are the same for the perfect code
$S_{{\mathcal H},\tau}'$.
\section{Rank and kernel of $S_{{\mathcal H},\tau}$}

 In the current section we discuss the ranks and the
dimensions of the kernels for the codes $S_{{\mathcal H},\tau}$.
We find the formulas for these invariants in terms of the
intersection of $\tau({\mathcal H})$ and ${\mathcal H}$. Note that
$\tau$ is an arbitrary bijection preserving ${\bf 0}$ in this
section with exception of Example 1.

 We denote the
dimension of a linear code $C$ by $\dim(C)$. The linear span of a
code $C$ over $F$ is denoted by $<C>$. The {\it rank of a code}
$C$, denoted by $\rank(C)$, is $\dim(<C>)$. The {\it kernel} $\Ker(C)$ of a code $C$ of
length $n$ is defined as the set of all vectors $x\in F^n$ such
that $x+C=C$.
Note that the all-zero vector is in $C$ if and only if
$\Ker(C)\subseteq C$.

Let $\tau$ be  a bijection from $F^r$ to $F^r$ (a permutation of
the coordinate positions of ${\mathcal H}$) that fixes the vector
${\bf 0}$. Define {\it the distension} of $\tau$ to be $
2^r-r-1-\dim({\mathcal H} \cap \tau({\mathcal H}))$, where
$\tau({\mathcal H})=\{\tau(x):x \in {\mathcal H}\}$. Note that
$2^r-r-1$ is the dimension of the code ${\mathcal H}$.


\begin{lemma}\label{Lemma 1}
Let $\tau$ be  a bijection from $F^r$ to $F^r$, $\tau({\bf
0})={\bf 0}$ and  $l$ be the distension of $\tau$. Then the rank
of $S_{{\mathcal H},\tau}$ is $(2^{r+1}-r-2)+l$.
\end{lemma}
\begin{proof}
Let $F_0^{2^r}$ denote the  even weight code of length $2^r$.

Since

$$<\bigcup_{a\in F^r} {\mathcal H}_a\times
{\mathcal H}_{\tau(a)}>= <\bigcup_{a\in F^r} ({\mathcal
H}+e_a+e_{\bf 0})\times ({\mathcal H}+e_{\tau(a)}+e_{\bf 0})>=$$
$$<{\mathcal H}\times {\mathcal H}\cup \{x|\tau(x):x \in
F_{0}^{2^r}\}>,$$
 we have \begin{equation}\label{lemma 1eq}<S_{{\mathcal H},\tau}>=<{\mathcal H}\times
{\mathcal H}\cup \{x|\tau(x):x \in F_{0}^{2^r}\}>.\end{equation}

 Let the vectors $\{u_i\}_{i\in
\{1,\ldots,2^r-1\}}$ be a basis
of $F_0^{2^r}$ such that $u_1,\ldots,u_{\dim({\mathcal H}\cap
\tau({\mathcal H}))}$ is a basis
of the subspace ${\mathcal H}\cap \tau^{-1}({\mathcal H})$ and
$u_1,\ldots,u_{2^r-r-1}$ is a basis
of ${\mathcal H}$.

Consider the following two sets:
$${\mathcal B}=\{u_i|u_j,1\leq i,j \leq 2^r-r-1\},$$
$${\mathcal B}'=\{u_i|\tau(u_i):\dim({\mathcal H}\cap
\tau({\mathcal H}))+1\leq i \leq 2^r-1\}.$$ We show that
${\mathcal B}\cup {\mathcal B}'$ is a basis
of $<S_{{\mathcal H},\tau}>$.

 The vectors of ${\mathcal B}\cup {\mathcal B}'$ are
linearly independent. Indeed, obviously, ${\mathcal B}$ is a basis
of ${\mathcal H}\times {\mathcal H}$. Suppose that a nonzero
vector $x|\tau(x)$ is spanned by ${\mathcal B}'$. Then by the
definition of ${\mathcal B}'$ the vector $x$ is not from
${\mathcal H}\cap \tau^{-1}({\mathcal H})$ and therefore $x$ and
$\tau(x)$ can not be simultaneously in ${\mathcal H}$, i.e.
$x|\tau(x)\notin {\mathcal H}\times {\mathcal H}=<{\mathcal B}>$.

We show that $<{\mathcal B} \cup {\mathcal B}'>=<S_{{\mathcal
H},\tau}>$. The equality (\ref{lemma 1eq}) and $<{\mathcal
B}>={\mathcal H}\times {\mathcal H}$ imply that it is sufficient
to prove that the vectors $u_i|\tau(u_i), i\in \{1,\ldots,2^r-1\}$
are in $<{\mathcal B}\cup{\mathcal B}'>$. By definition of
$\{u_i\}_{i\in \{1,\ldots,\dim({\mathcal H}\cap \tau({\mathcal
H}))\}}$ these vectors are in ${\mathcal H}\cap
\tau^{-1}({\mathcal H})$, so $\tau(u_i)$ is in ${\mathcal H}$ for
$i\in \{1,\ldots,\dim({\mathcal H}\cap \tau({\mathcal
H}))\}$. We see that in this case the vector $u_i|\tau(u_i)$ is in
$<{\mathcal B}>={\mathcal H}\times {\mathcal H}$ for any $i\in
\{1,\ldots,\dim({\mathcal H}\cap \tau({\mathcal H}))\}$. The
remaining vectors $u_i|\tau(u_i),$  $i\in \{\dim({\mathcal
H}\cap \tau({\mathcal H}))+1,\ldots,2^r-1\}$ are from ${\mathcal
B}'$, so ${\mathcal B} \cup {\mathcal B}'$  spans $<S_{{\mathcal
H},\tau}>$.

The rank of $S_{{\mathcal H},\tau}$ is $|{\mathcal B}|+|{\mathcal
B'}|=2(2^r-r-1)+(2^r-1-\dim({\mathcal H}\cap \tau({\mathcal
H})))=(2^{r+1}-r-2)+2^r-r-1-\dim ({\mathcal H}\cap
\tau({\mathcal H})).$ Taking into account that $2^r-r-1-\dim
({\mathcal H}\cap \tau({\mathcal H}))$ is the distension of
$\tau$, we obtain the required.

\end{proof}

For $a\in F^r$ the set of the supports of the codewords of ${\mathcal H}'$ of weight $3$ with ones in the coordinate
indexed by $a$
 is called a {\it star} of ${\mathcal H}'$ and  denoted  by
$\mbox{Star}(a)$, i.e. we have
$$\mbox{Star}(a)=\{\{a,b,a+b\}:b\in F^r\setminus ({\bf 0}\cup a)\}.$$

Let $\tau$ be  a bijection from $F^r$ to $F^r$ such that
$\tau({\bf 0})={\bf 0}$. Since we always have $\tau({\bf 0})={\bf
0}$, we can consider that $\tau$ acts on the coordinate positions
of ${\mathcal H}'$ which are indexed by the nonzero vectors of
$F^r$ and use notation $\tau({\mathcal H}')$ throughout the text.

 Note that $\tau(\mbox{Star}(a))$ is a star of ${\mathcal H}'$
if and only if  it is $\mbox{Star}(\tau(a))$, which is equivalent
to $\tau(a+c)=\tau(a)+\tau(c)$ for all $c\in F^r$. If
$\tau(a+c)=\tau(a)+\tau(c),$ $\tau(b+c)=\tau(b)+\tau(c)$ for all
$c\in F^r$, then
$\tau(a+b+c)=\tau(a)+\tau(b+c)=\tau(a)+\tau(b)+\tau(c)=\tau(a+b)+\tau(c).$
 We conclude that if $\tau(\mbox{Star}(a))$ and
$\tau(\mbox{Star}(b))$ are stars of ${\mathcal H'}$, then
$\tau(\mbox{Star}(a+b))$ is a star of ${\mathcal H'}$. This
implies that
 the number of stars of
${\mathcal H}'$ that are mapped to stars of ${\mathcal H}'$ by
$\tau$ is always a power of two but one. Define the {\it
deficiency} of $\tau$ to be $r-\log(|\{a\in
F^r:\tau(\mbox{Star}(a))=\mbox{Star}(\tau(a))\}|+1)$.

\begin{lemma}\label{Lemma 2}
Let $\tau$ be  a bijection from $F^r$ to $F^r$ such that
$\tau({\bf 0})={\bf 0}$ and the deficiency of  $\tau$ be $k$. Then
we have
$$\mathrm{dim(Ker}(S_{{\mathcal H},\tau}))=
2^{r+1}-r-2-k.$$
\end{lemma}
\begin{proof}
It is easy to see that ${\mathcal H}\times {\mathcal H}\leq
\Ker(S_{{\mathcal H},\tau})$. We have that
\begin{equation}\label{kernel} \dim({\mathcal H}\times {\mathcal H})=2^{r+1}-2r-2.
\end{equation} Let us consider the codeword $e_{\bf 0} + e_a| e_{\bf 0} + e_{\tau(a)} \in
S_{{\mathcal H},\tau} \setminus {\mathcal H}\times {\mathcal H}$ for any $a \in F^r$.
We show that it belongs to $\Ker(S_{{\mathcal H},\tau})$ if and only if
$\tau(\mbox{Star}(a))$ is a star of ${\mathcal H}'$.  We have
 $$(e_{\bf 0} + e_a| e_{\bf 0} + e_{\tau(a)}) +
S_{{\mathcal H},\tau}$$  $$=\bigcup_{b\in F^r} (e_{\bf 0} + e_a +
{\mathcal H}_b)\times (e_{\bf 0} + e_{\tau(a)} + {\mathcal
H}_{\tau(b)})$$
$$= \bigcup_{b\in F^r} {\mathcal H}_{a+b}\times
{\mathcal H}_{\tau(a)+\tau(b)}.$$ The equality
$$ \bigcup_{b\in
F^r} {\mathcal H}_{a+b}\times {\mathcal H}_{\tau(a)+\tau(b)} =
\bigcup_{c\in F^r} {\mathcal H}_{c}\times {\mathcal H}_{\tau(c)}=
S_{{\mathcal H},\tau}$$ holds if and only if
$\tau(a)+\tau(b)=\tau(a+b)$ for any $b\in F^r$, i.e.
$\tau(\mbox{Star}(a))$ is a star of ${\mathcal H}'$. Taking into
account (\ref{kernel}) we have $$\dim(\Ker(S_{{\mathcal
H},\tau}))=\dim({\mathcal H}\times {\mathcal H}) +
\log(|\{a:\tau(\mbox{Star}(a))=\mbox{Star}(\tau(a))|+1)=$$
$$(2^{r+1}-2r-2)+(r-k)=2^{r+1}-r-2-k.$$
\end{proof}

 From Lemma \ref{Lemma 2} we see that the dimension
of the kernel of a code $S_{{\mathcal H},\tau}$ of length
$2^{r+1}-r-2$ is at least $2^{r+1}-2r-2$. If the dimension of the
kernel of $S_{{\mathcal H},\tau}$ is $2^{r+1}-2r-2$ then the code
$S_{{\mathcal H},\tau}$ could not be obtained by the Mollard
construction for propelinear codes with large ranks, see
\cite{BMRS}. Recall that two codes of length $n$ are called {\it
equivalent} if there is an automorphism of $F^n$ that maps one
code to another.

\begin{theorem}\label{Prnew}
Let $\tau$ be  a bijection from $F^r$ to $F^r$ such that
$\tau({\bf 0})={\bf 0}$ and $r$ be the deficiency of  $\tau$. Then
the codes $S_{{\mathcal H},\tau}$ and $S'_{{\mathcal H},\tau}$ are
not equivalent to any extended perfect Mollard code and perfect
Mollard code respectively.
\end{theorem}
\begin{proof}
By the condition of the theorem the  deficiency of the bijection
$\tau$ from $F^r$ to $F^r$ is $r$ so by Lemma \ref{Lemma 2}  the
dimension of $\Ker(S_{{\mathcal H},\tau})$ is $2(2^{r}-r-1)$.
 Since $\dim({\mathcal H})=2^{r}-r-1$, from the proof of Lemma
\ref{Lemma 2} we have  ${\mathcal H}\times {\mathcal H}\leq
\Ker(S_{{\mathcal H},\tau})$, so $\Ker(S_{{\mathcal
H},\tau})={\mathcal H}\times {\mathcal H}$. We see that the
codewords of $\Ker(S_{{\mathcal H},\tau})={\mathcal H}\times
{\mathcal H}$ of weight 4 are either $c|{\bf 0}$ or ${\bf 0}|c$
for all codewords $c$ of ${\mathcal H}$ of weight 4. In
particular, for any fixed coordinate there are $2^r$ coordinates
such that there are no codewords from $\Ker(S_{{\mathcal H},\tau})$
of weight 4 with ones in any of these positions and the fixed
coordinate simultaneously.

Consider the construction for Mollard codes. Let $C$ and $D$ be
any two perfect codes of lengths $t$ and $m$, respectively,
containing all-zero vectors. We index the coordinates of
 $F^t$ by pairs $(i,0)$, $i\in \{1,\ldots,t\}$, the coordinates of $F^m$ by pairs $(0,j)$, $j\in
 \{1,\ldots,m\}$, the coordinates of  $F^{tm}$ by pairs $(i,j)$, $i\in
\{1,\ldots,t\}$, $j\in \{1,\ldots,m\}$ and the coordinates of the
Mollard code are indexed by pairs $(i,j)$, $i\in \{0,\ldots,t\}$
$j\in \{0,\ldots,m\}$, where $i$ and $j$ are not $0$
simultaneously.

Let \, $x=(x_{(1,1)},x_{(1,2)}, \ldots,x_{(1,m)},
x_{(2,1)},\ldots,x_{(2,m)},\ldots,x_{(t,1)},\ldots,x_{(t,m)})$ \,
be a vector from $F^{tm}.$
 The
generalized parity check functions \, $p_{1}(x):F^{tm}\rightarrow
F^t$ and $p_{2}(x):F^{tm}\rightarrow F^m$ are defined as
$$(p_{1}(x))_{(i,0)}=\sum_{j=1}^{m}x_{(i,j)}, \,\, (p_{2}(x))_{(0,j)}=\sum_{i=1}^{t}x_{(i,j)}.$$

 Let $f$ be any function from
$C$ to $F^m$.  Denote by $Z$ the following set
$$\{(x,p_1(x),p_2(x)): x \in F^{tm}\}.$$
 The perfect binary Mollard code of length
$tm+t+m$, see \cite{Mollard} consists of cosets of $Z$
$$M(C,D)=\{({\bf 0}|y|z\, + \, f(y)) : y\in\,C, z
\in\,D\}+Z.$$ Obviously,
$Z$ is a subspace of $\Ker(M(C,D))$. Moreover, the kernel of the
extended Mollard code contains the extension $\overline{Z}$ of
$Z$:
$$\overline{Z}=\{(x|p_1(x)|p_2(x)|\wt(x)\mbox{ mod }2):x \in
F^{tm}\}.$$

We index the last coordinate position of $\overline{Z}$ by
$(0,0)$.

It is easy to see that for any two coordinates $\overline{Z}$
contains a unique codeword of weight 4 with ones in these
coordinates. Indeed, the set of supports of the codewords of
weight 4 from $\overline{Z}$ is

\medskip
\noindent $\{\{(i_1,j_1), \, (i_2,j_2), \, (i_2,j_1), \,
(i_1,j_2)\}:
 i_1, \, i_2\in \{0,\ldots,m\}, j_1, \,
j_2\in \{0,\ldots,t\}, i_1\neq i_2, j_1\neq j_2\}$.

\medskip
Suppose that the code $S_{{\mathcal H},\tau}$ is equivalent to the
extended perfect Mollard code $\overline{M(C,D)}$ via an
automorphism of the Hamming space. Then $\Ker(S_{{\mathcal
H},\tau})$ is necessarily equivalent to $\Ker(\overline{M(C,D)})$.
For any two coordinates of $\overline{Z}$ there is at least one
codeword of weight 4 in $\overline{Z}, \overline{Z}\leq
\Ker(\overline{M(C,D)})$ with ones in these coordinates. On the
other hand the considerations in the beginning of the proof of the
theorem imply that there are pairs of coordinate   with no vectors
from $\Ker(S_{{\mathcal H},\tau})$ of weight 4 with ones in these
coordinates. So, $S_{{\mathcal H},\tau}$ is not a Mollard code.

We now turn to the case of punctured codes. Consider the
puncturing $S_{{\mathcal H},\tau}'$ of $S_{{\mathcal H},\tau}$
defined by (\ref{code_S}).

The codewords of $\Ker(S'_{{\mathcal H},\tau})$ are obtained from
the codewords of $\Ker(S_{{\mathcal H},\tau})={\mathcal H}\times
{\mathcal H}$ by puncturing, so the codewords from
$\Ker(S'_{{\mathcal H},\tau})$ of weight 3 are $\{c|{\bf 0}: c\in
{\mathcal H}'\}$,  $\wt(c)=3$. Therefore there are at least $2^r$
coordinates of $S'_{{\mathcal H},\tau}$ that are zeros for all
codewords of weight 3 in $\Ker(S_{{\mathcal H},\tau})$. On the
other hand, the kernel of the Mollard code $M(C,D)$ contains
$Z=\{(x|p_1(x)|p_2(x)):x \in F^{tm}\}$. Then,
$\{\{(i,j),(i,0),(0,j)\}:i\in \{1,\ldots,t\},j\in\{1,\ldots,m\}\}$
are supports of some codewords of weight 3 from $\Ker(M(C,D))$.
We see that for every coordinate of $M(C,D)$ there is at least one
codeword of weight 3 from $\Ker(M(C,D))$ with one in this
coordinate. Since this property does not hold for $\Ker(S'_{{\mathcal H},\tau})$, we conclude that $\Ker(S'_{{\mathcal H},\tau})$ is not
the kernel of any Mollard code.

\end{proof}

{\bf Example 1.} Recall that {\it the dihedral group} $D_m$ is the
group formed by the symmetries
 of $m$-sized polygon.

 Consider a group that is generated by an element
 $\alpha$ of order $m$ and $\beta$ of order 2 that satisfy the
 relation $\beta\alpha\beta=\alpha^{-1}$. It is well-known that the group is isomorphic to the dihedral group $D_m$.
 Consider
  the mapping
 $T$  that fixes any element of the subgroup generated by
 $\alpha$ and sends $\beta\alpha^{i}$ to $\beta\alpha^{i+1}$, $i \in \{0,\ldots, m-1\}$. We show that $T$ is an automorphism of
 the group. Indeed, using the
generator relation  we see that
$(\beta\alpha)^2=\beta\alpha\beta\alpha=\alpha\alpha^{-1}=1$, so
 $\beta\alpha$ has order two. Moreover, we have
 $(\beta\alpha)\alpha(\beta\alpha)=(\beta\alpha^2\beta)\alpha=\alpha^{-2}\alpha=\alpha^{-1}.$
 We conclude that $T$ is an automorphism because the generator
 relation for the dihedral group for $\alpha$ and involution $\beta\alpha$ is
 fulfilled.

We now consider the regular subgroup of $\GA(r,2)$ from \cite{H}.
Let $\alpha$ be $((101),A)$ and $\beta$ be $((001),B)$, where
$A=\left(%
\begin{array}{ccc}
  0 & 1 & 0 \\
  1 & 0 & 0 \\
  1 & 0 & 1 \\
\end{array}%
\right)$, $B=\left(%
\begin{array}{ccc}
  0 & 1 & 0 \\
  1 & 0 & 0 \\
  0 & 0 & 1 \\
\end{array}%
\right)$.

We see that the orders of $\alpha$ and $\beta$ are 4 and 2
respectively. Moreover $\beta\alpha=((010),BA)$, where
$BA=\left(%
\begin{array}{ccc}
  1 & 0 & 0 \\
  0 & 1 & 0 \\
  1 & 0 & 1 \\
\end{array}%
\right)$  has order two, so $\beta\alpha\beta=\alpha^{-1}$ and the
group generated by $\alpha$ and $\beta$ is isomorphic to $D_4$. We
have the following: $\beta\alpha^2=((110),BA^2)$ and
$\beta\alpha^3=((100),BA^3)$. Consider the automorphism $T$ for
dihedral groups described above:
 $T$ fixes
 $\alpha^i$ and sends $\beta\alpha^{i}$ to $\beta\alpha^{i+1}$, $i \in \{0,1,2, 3\}$.
 Let $\tau$ be the permutation induced by the automorphism $T$.
Since $\beta=((001),B)$, $\beta\alpha=((010),BA)$,
$\beta\alpha^2=((110),BA^2)$, $\beta\alpha^3=((100),BA^3)$ the
bijection $\tau$ shifts the following vectors in the cyclic order
$(001)$, $(010)$, $(110)$, $(100)$ and fixes any other vector from
$F^3$.

 Let ${\mathcal H}$ be the code with coordinates indexed by
vectors of $F^3$ in the lexicographical order and numbers
$\{0,\ldots,7\}$ in the ascending order. Then $\tau$ is the
permutation $(4,3,2,1)$. We have the following supports of the
codewords in ${\mathcal H}$ containing 0:
\medskip

\noindent
$\{0,1,2,3\},\{0,1,4,5\},\{0,1,6,7\},\{0,2,4,6\},\{0,2,5,7\},\{0,3,4,7\},\{0,3,5,6\},\\
\{0,\ldots, 7\}.$
\medskip

Since $\tau=(4,2,3,1)$, we have the following supports of
codewords  of $\tau(\mathcal H)$ containing 0:
\medskip

\noindent
$\{0,1,3,4\},\{0,2,4,5\},\{0,4,6,7\},\{0,2,3,6\},\{0,3,5,7\},\{0,1,2,7\},\{0,1,5,6\},\\
\{0,\ldots, 7\}.$
\medskip

 The supports of the codewords of ${\mathcal H}$
and $\tau(\mathcal H)$ not containing $0$ are the comple\-ments of
those that contain $0$ to $\{0,\ldots,7\}$.
 Then $\tau({\mathcal H})\cap {\mathcal H}$
consists of the all-zero and the all-one vectors, so $\tau$ has
the distension 3. The deficiency of $\tau$ is 3 because the codes
$\tau({\mathcal H'})$ and ${\mathcal H'}$ do not have common
codewords of weight 3. By Lemmas \ref{Lemma 1} and \ref{Lemma 2}
the code $S_{{\mathcal H},\tau}$ is a propelinear extended perfect
code of length 16, rank 14 and the kernel dimension  8.

\section{Infinite series of new propelinear perfect codes}

 In this section we construct an infinite series of
propelinear codes $S_{\mathcal H,\tau}$ of prefull rank and the
dimension of the kernel $2^{r+1}-2r-2$, i.e. the maximum rank and
the minimum dimension of kernel that we can obtain by the
considered construction. In view of Theorem \ref{Prnew} these
codes are new propelinear codes.

\begin{lemma}\label{Lemma 3}
 Let $\tau$ be  a bijection from $F^{r_1}$ to $F^{r_1}$ such that
$\tau({\bf 0})={\bf 0}$ and $\sigma$ be  a bijection from
$F^{r_2}$ to $F^{r_2}$, $\sigma({\bf 0})={\bf 0}$  with the
distensions $l_1$ and $l_2$  and the deficiencies $k_1$ and $k_2$
respectively. Then the  bijection $\tau|\sigma$ from $F^{r_1+r_2}$
to $F^{r_1+r_2}$ has the distension $l_1+l_2$ and the deficiency
$k_1+k_2$.

\end{lemma}

\begin{proof}

Consider a codeword $x$ of the extended Hamming code ${\mathcal
H}$ of length $2^{r_1+r_2}$, whose coordinates are indexed by the
vectors $a|b$, $a\in F^{r_1}$, $b\in F^{r_2}$. The bijection
$\tau|\sigma$ acts on the vectors of  $F^{r_1+r_2}$, so it could
be treated as a permutation on the coordinate positions of
${\mathcal H}$.

A vector $x\in F^{2^{r_1+r_2}}$ is in ${\mathcal H}$ if and only
if
$$\mbox{supp}(x)=\{a_i|b_i: i \in
\{1,\ldots,\wt(x)\}\}\mbox{ is such that
}\sum_{i\in\{1,\ldots,\wt(x)\}}a_i|b_i={\bf 0}.$$


The vector $(\tau|\sigma)(x)$ with the support
$\{\tau(a_i)|\sigma(b_i):i\in\{1,\ldots,\wt(x)\}\}$ is in
${\mathcal H}$ if and only if
$$\sum_{i\in\{1,\ldots,\wt(x)\}}\tau(a_i)={\bf 0}\mbox{ and }
\sum_{i\in\{1,\ldots,\wt(x)\}}\sigma(b_i)={\bf 0}.$$ In other
words, $(\tau|\sigma)(x)$ is in ${\mathcal H}$ if and only if the
vectors with the supports $\{\tau(a_i):i\in\{1,\ldots,\wt(x)\}\}$
and $\{\sigma(b_i):i\in\{1,\ldots,\wt(x)\}\}$ are codewords of
${\mathcal H}^{r_1}$ and ${\mathcal H}^{r_2}$ respectively. We
conclude that $\dim((\tau|\sigma)({\mathcal H})\cap {\mathcal
H})=\dim(\tau({\mathcal H}^{r_1})\cap {\mathcal
H}^{r_1})+\dim(\sigma({\mathcal H}^{r_2})\cap {\mathcal H}^{r_2})
,$ so the distension of $\tau|\sigma$ is $l_1+l_2$.

For a nonzero vector $a|b\in F^{r_1+r_2}$ consider
$\mbox{Star}(a|b)$:
$$ \{\{a|b,c'|c'',a+c'|b+c''\}:(c',c'')\in F^{r_1+r_2}\setminus (a|b\cup {\bf 0})\}.$$
So, $(\tau|\sigma)(\mbox{Star}(a|b))$ is
$$ \{\{\tau(a)|\sigma(b),\tau(c')|\sigma(c''),\tau(a+c')|\sigma(b+c'')\}:(c',c'')\in F^{r_1+r_2}\setminus (a|b\cup {\bf 0})\}.$$

From this expression we see that the set
$(\tau|\sigma)(\mbox{Star}(a|b))$ is a star of ${\mathcal H'}$ if
and only if $\tau(a)+\tau(c')+\tau(a+c')={\bf 0}$ and
$\sigma(b)+\sigma(c'')+\sigma(b+c'')={\bf 0}$ for $c'\in F^{r_1},
c''\in F^{r_2}$. In other words, we have
$(\tau|\sigma)(\mbox{Star}(a|b))\subset{\mathcal H}'$ if and only
if
$$a,b\neq {\bf 0} \mbox{ and } \tau(\mbox{Star}(a))\subset ({\mathcal H}^{r_1})'\mbox{ and }
\sigma(\mbox{Star}(b))\subset ({\mathcal H}^{r_2})'\mbox{ or }$$
$$a={\bf 0},\mbox{ }b\neq {\bf 0}\mbox{ and }\sigma(\mbox{Star}(b))\subset ({\mathcal
H}^{r_2})'\mbox{ or}$$
$$b={\bf 0}, a\neq {\bf 0}\mbox { and }
\tau(\mbox{Star}(a))\subset ({\mathcal H}^{r_1})'.$$

 We conclude that there are total $(2^{r_1-k_1}-1)(2^{r_2-k_2}-1)+2^{r_1-k_1}-1+2^{r_2-k_2}-1=2^{r_1+r_2-k_1-k_2}-1$
 stars in ${\mathcal H}'$
 that are mapped to stars in ${\mathcal H}'$ by $\tau|\sigma$.
 So, the deficiency of $\tau|\sigma$ is $k_1+k_2$.

\end{proof}

{\bf Direct product construction for regular subgroups of the
general affine group \cite{Mog}.} Let $G_1$ and $G_2$ be regular
subgroups of $\GA(r_1,2)$ and $\GA(r_2,2)$
respectively. Given two elements $g_1=(a,A)\in G_1$ and
$g_2=(b,B)\in G_2$ consider the following element of
$\GA(r_1+r_2,2)$, which we denote by $g_1|g_2$:
$$(a|b,\left(%
\begin{array}{cc}
  A &  {\bf 0}_{r_1,r_2} \\
 {\bf 0}_{r_2, r_1} & B \\
\end{array}%
\right)),$$ here ${\bf 0}_{r_1,r_2}$ and
 ${\bf 0}_{r_2, r_1}$ are the all-zero $r_1\times r_2$ and  $r_2\times
 r_1$
matrices respectively. It is easy to see that $\{g_1|g_2:g_1\in
G_1, g_2\in G_2\}$ is a regular subgroup of $\GA(r_1+r_2,2)$. We
denote this group by $G_1 \otimes G_2$.

 Consider automorphisms $T$ and $S$ of $G_1$ and
$G_2$ respectively with the induced permutations $\tau$ and
$\sigma$ respectively. Define the following permutation $T|S$ of the
elements of  $G_1\otimes G_2$: $(T|S)(g_1|g_2)=T(g_1)|S(g_2)$.
Obviously, $T|S$ is an automorphism of the group $G_1\otimes G_2$
and the permutation  induced by $T|S$ is
$\tau|\sigma$.

\begin{theorem}\label{Theorem 2}
Let  $T$ and $S$ be automorphisms of regular subgroups $G_1$ and
$G_2$ of $\GA(r_1,2)$ and $\GA(r_2,2)$ respectively. Let $\tau$ and
$\sigma$ be the
 permutations induced by $T$ and $S$  with the distensions $l_1$ and $l_2$  and the
deficiencies $k_1$ and $k_2$ respectively. Then $S_{{\mathcal
H},\tau|\sigma}$ is a propelinear extended perfect code of length
$2^{r_1+r_2+1}$ with  $\rank(S_{{\mathcal H},\tau|\sigma})=
2^{r_1+r_2+1}-r_1-r_2-2+l_1+l_2$ and $\dim(\Ker(S_{{\mathcal
H},\tau}))= 2^{r_1+r_2+1}-r_1-r_2-2-k_1-k_2$.

\end{theorem}
\begin{proof}
The regular subgroup $G_1 \otimes G_2$ of $\GA(r_1+r_2,2)$ obtained by direct
product construction has the automorphism $T|S$ with the induced
permutation $\tau|\sigma$, so $S_{{\mathcal H},\tau|\sigma}$ is a
propelinear extended perfect code.  The desired values for the
rank and the dimension of the kernel of the code $S_{{\mathcal
H},\tau|\sigma}$ follow from Lemmas \ref{Lemma 1}, \ref{Lemma 2}
and \ref{Lemma 3}.

\end{proof}

 Table 1 contains
the values for the distension and the deficiency for the
permuta\-tions induced by the  automorphisms of the regular
subgroups of $\GA(r,2)$ for $3\leq r\leq 5$. This result was
obtained by computer. \begin{table}[h]
\noindent\begin{tabular}{|c|c|}
  \hline
   & The values for $(l,k)$, where $l$ is the distension and \\
   & $k$ is the deficiency of the permutations induced by \\
   &  the automorphisms of the regular subgroups of $\GA(r,2)$ \\

  \hline
  $r=3$ & (0,0),(1,2),(2,3),(3,3) \\
  $r=4$ &  (0,0),(1,2),(2,3),(2,4),(3,3),(3,4),(4,3)\\
  $r=5$ &  (0,0),(1,2),(1,4),(2,3),(2,4),(2,5),(3,3),(3,4),(3,5),(4,3),(4,4),(4,5),(5,4),(5,5)\\
  \hline
\end{tabular}
\caption{}
\end{table}
For any length $n$, $n\geq 16$ Theorem \ref{Theorem 2} and the
data from Table 1 imply the existence of propelinear extended
perfect codes of length $n$ of varying rank and kernel and, in
particular, codes of prefull rank $n-1$ and codes with the
dimension of the kernel $n-2\log n -2$.

\begin{corollary}\label{Corollary 1}
For any $r\equiv0,2 \pmod 3$, $r\geq 3$ there is a propelinear
extended perfect code $S_{{\mathcal H},\tau}$ of length $2^{r+1}$,
$\rank(S_{{\mathcal H},\tau}) = 2^{r+1}-2$ and
$\dim(\Ker(S_{{\mathcal H},\tau}))=2^{r+1}-2r-2$. For any
$r\equiv1\pmod 3$, $r\geq 3$ there is a propelinear extended
perfect code $S_{{\mathcal H},\tau}$
 of length $2^{r+1}$, $\rank(S_{{\mathcal H},\tau}) = 2^{r+1}-2$ with $\dim(\Ker(S_{{\mathcal H},\tau}))=
2^{r+1}-2r-1$ and there is a propelinear extended perfect code
$S_{{\mathcal H},\tau}$ of $\rank(S_{{\mathcal H},\tau}) =
2^{r+1}-3$ and $\dim(\Ker(S_{{\mathcal
H},\tau}))=2^{r+1}-2r-2$.

\end{corollary}
\begin{proof}

We fix $\tau:F^3\rightarrow F^3$ to be the permutation induced by
the automorphism of the regular subgroup $G_1$ considered in the
Example 1. The permutation $\tau$ has
the distension and the deficiency 3.
 We vary the permutation $\sigma$ among the
permutations that are induced by the automorphisms of the regular
subgroups $G_2$ of $\GA(r_2,2)$ for $r_2: 3\leq r_2\leq 5$ with the
maximum distensions and deficiencies. According to Example 1 and
Table 1, for $r_2=3$ or $5$ there are permutations with both the
distension and the deficiency equal $3$ or $5$ respectively. For
$r_2=4$ there are permutations with the distension 4 and the
deficiency 3 and the distension 3 and the deficiency 4.

 Let $r$ be $3m+r_2$ for some
$m\geq 0$ and $3\leq r_2\leq 5$. There is a
 regular subgroup of $\GA(3m,2)$ which is the direct product of the
 $m$ copies of $G_1$: $G_1\otimes\ldots \otimes G_1$ with the
 permutation $\tau|\ldots|\tau$ induced by the  automorphism
 $T|\ldots|T$, where $T$ is the  automorphism from
 Example 1. Using the direct product construction again we obtain a
 regular subgroup $G_1\otimes\ldots \otimes G_1\otimes G_2$ of $\GA(3m+r_2,2)$ that has an automorphism with the induced permutation
 $\tau|\ldots|\tau|\sigma$. By Theorem \ref{Theorem 2} for $r\equiv0,2 \pmod 3$ the code $S_{{\mathcal H},\tau|\ldots|\tau|\sigma}$ of length $2^{r+1}$ has
 rank $2^{r+1}-1$ and $\dim(\Ker(S_{{\mathcal
H},\tau|\ldots|\tau|\sigma)})=2^{r+1}-2r-2$ and for $r\equiv1\pmod
3$ the code $S_{{\mathcal H},\tau|\ldots|\tau|\sigma}$ has rank
$2^{r+1}-2$ and $\dim(\Ker(S_{{\mathcal
H},\tau|\ldots|\tau|\sigma})=2^{r+1}-2r-1$ or rank $2^{r+1}-3$ and
$\dim(\Ker(S_{{\mathcal
H},\tau|\ldots|\tau|\sigma}))=2^{r+1}-2r-2$.
\end{proof}

{\bf Remark 2.} By  Corollary \ref{Corollary 1} there are the
propelinear codes of length $2^{r+1}-1$ with the dimension of the
kernel $2^{r+1}-2r-2$. Taking into account Theorem \ref{Prnew} and
Lemma 2, for any
 $r\geq 3$ there are
propelinear perfect codes of length $2^{r+1}-1$ that could not be
obtained by Mollard construction.

{\bf Remark 3.} Applying Theorem \ref{Theorem 2} iteratively we
obtain a relatively large class of nonequiva\-lent propelinear
perfect codes of any admissible length more than 7.  For example,
using data from Table 1, we obtained propelinear perfect codes of
length $511$  with the size of kernels more than 497 of all
possible ranks with the only exception of the full rank. Among
these codes at least 32 codes are pairwise nonequivalent as the
calculated values for the pairs of rank and the dimension of the
kernel  are different. Note that there are 51 different pairs of
the rank and the dimension of kernel for the perfect codes of
length 511 with the  dimension of the kernel at least 497, see 
\cite{AHS}.


\bigskip
\begin{thebibliography}{1}
\bibitem{AHS}
 S.~V.~Avgustinovich,  O. Heden,  F.~I.~Solov'eva,
Perfect codes of full rank with big kernels, {\em Discrete
Analysis and Oper. Research}, 1 {\bf 8}(4) (2001), 3--8 (in
Russian). 

\bibitem{BMRS}
      J.~Borges, I.~Yu.~Mogilnykh, J.~Rif\`{a},  F.~I.~Solov'eva,  Structural properties of binary propelinear codes,
       {\em Advances in Mathematics of Communication}, {\bf 6} (2012),
       329--346. 

\bibitem{BorgesMogilnykhRifaSoloveva}
      J.~Borges,   I.~Yu.~Mogilnykh, J.~Rif\`{a},   F.~I.~Solov'eva, On the number of nonequivalent propelinear extended perfect codes. {\em The Electronic Journal of
      Combinatorics},
      {\bf 20}(2)  (2013),  37--50.


\bibitem{BPR} J.~Borges,  K.~T.~Phelps,  J.~Rif\`{a},   The  rank  and
kernel  of  extended  1-perfect Z4-linear  and  additive
non-Z4-linear  codes, {\em IEEE Transactions on Information
Theory}, {\bf 49} (2003),
 2028--2034. 

\bibitem{BR}
J.~Borges,   J.~Rif\`{a},  A characterization of 1-perfect
additive codes, {\em IEEE Transactions on Information Theory},
{\bf 45} (1999), 1688--1697. 



\bibitem{GMS} G.~K.~Guskov, I.~Yu.~Mogilnykh, F.~I.~Solov'eva, Ranks of propelinear perfect binary
codes, {\em Siberian Electronic Mathematical Reports}, {\bf 10}
(2013), 443--449. 

\bibitem{HKCSS}  A.~R.~Hammons, Jr, P.~V.~Kumar, A.~R.~Calderbank,
N.~J.~A.~Sloane and  P.~Sol$\acute{e}$,  The Z4 -linearity of
Kerdock, Preparata, Goethals,  and related  codes, {\em IEEE
Transactions on Information Theory},  {\bf  40}(2)  (1994),
301--319. 

\bibitem{H} P.~Hegedus,   Regular
subgroups of the Affine group, {\em Journal of Algebra}, {\bf 225}
(2000), 740--742. 

\bibitem{K1} D.~S.~Krotov, Z4-linear  Hadamard  and  extended  perfect  codes, WCC2001,
 International  Workshop  on  Coding  and  Cryptography,
 {\em
 Electronic Notes in Discrete Mathematics}, {\bf 6} (2001),
 107--112. 

\bibitem{K2} D.~S.~Krotov, Z2k-Dual Binary Codes, {\em
IEEE Transactions on Information Theory}, {\bf 53} (2007),
1532--1537. 

\bibitem{KrotovPotapov}
 D.~S.~Krotov,   V.~N.~Potapov,  Propelinear 1-perfect codes from quadratic functions, {\em
 IEEE Transactions on
Information Theory}, {\bf 60}  (2014), 2065--2068. 

\bibitem{PRS}
 J.~Pujol,  J.~Rif\`{a},  F.~I.~Solov'eva, Construction of Z4-Linear
Reed-Muller Codes. {\em IEEE Transactions on Information Theory},
{\bf 55}(1) (2009), 99--104. 

\bibitem{MS}  F.~J.~MacWilliams and N.~J.~A.~Sloane,
{\it  The Theory of Error-Correcting Codes}, North-Holland
Publishing Company, 1977. 



\bibitem{Mog} I.~Yu.~Mogilnykh,  A note on regular subgroups of the automorphism group of the linear Hadamard
code, {\em Siberian Electronic Mathematical Reports}, {\bf 15}
(2018), 1455--1462. 


\bibitem{MS2} I.~Yu.~Mogilnykh, F.~I.~Solov'eva, On separability of the classes of homogeneous and transitive perfect binary
codes, {\em Problems of Information Transmission}, {\bf 51}
(2015), 139--147. 


\bibitem{Mollard}
M.~Mollard,  A generalized parity function and its use in the
construction of perfect codes, {\em SIAM J. Alg. Discrete Math.},
{\bf
 7}(1) (1986), 113--115. 

\bibitem{OPP}
  P.~R.~J.~\"{O}sterg{\aa}rd,  K.~T.~Phelps,  O.~Pottonen, The perfect binary one-error-correcting codes of length 15: Part
 II-properties, {\em IEEE Transactions on Information Theory } {\bf  56}   (2010),
2571\nobreakdash--2582. 

\bibitem{Phelps}
   K.~T.~Phelps,  A combinatorial construction of perfect codes,
{\em SIAM J. Alg. Disc. Meth.}, {\bf  4} (1983), 398--403. 

\bibitem{PR} K.~T.~Phelps,   J.~Rif\`{a},  On binary 1-perfect
additive codes: some structural properties, {\em IEEE Transactions
on Information Theory},  {\bf 48}  (2002), 2587--2592.



\bibitem{RP} J.~Rif\`{a}, J.~Pujol,  Translation-invariant propelinear
codes, {\em IEEE Transactions on Information Theory},  {\bf 43}
(1997), 590--598. Zbl 0876.94040

\bibitem{Pot}  V.~N.~Potapov, A lower bound for the number of transitive perfect codes,
{\em Journal of Application and Industrial Mathematics},  {\bf
1}(3) (2007), 373\nobreakdash--379.  
\bibitem{Sol1981} F.~I.~Solov'eva, On binary nongroup codes, {\em Methody Discretnogo
Analiza}, {\bf 37} (1981), 65--75 (in Russian). 


\end{thebibliography}
\end{document}